\documentclass[12pt,oneside]{amsart}

\usepackage{amsfonts}
\usepackage{amsmath, amsthm, amssymb, ulem, amscd}

\def\crn#1#2{{\vcenter{\vbox{
        \hbox{\kern#2pt \vrule width.#2pt height#1pt
           }
          \hrule height.#2pt}}}}

\newcommand{\Ad}{\operatorname{Ad}}

\newcommand{\Aut}{\operatorname{Aut}}

\renewcommand{\span}{\operatorname{span}}
\renewcommand{\line}{{\operatorname{line}}}
\newcommand{\ray}{{\operatorname{ray}}}
\newcommand{\gr}{\operatorname{gr}}

\newcommand{\tf}{\operatorname{tf}}

\newcommand{\R}{\mathbb R}
\newcommand{\V}{\mathbb V}
\renewcommand{\P}{\mathbb P}

\newcommand{\Z}{\mathbb Z}

\newcommand{\Up}{\Upsilon}

\newcommand{\gh}{\widehat{g}}

\newcommand{\wh}{\widehat}

\newcommand{\nt}{\widetilde{\nabla}}

\newcommand{\gt}{\widetilde{g}}

\newcommand{\cF}{\mathcal{F}}

\newcommand{\cE}{\mathcal{E}}
\newcommand{\cG}{\mathcal{G}}
\newcommand{\cGt}{\widetilde{\mathcal{G}}} 

\newcommand{\cB}{\mathcal{B}}

\newcommand{\cQ}{\mathcal{Q}}

\newcommand{\cV}{\mathcal{V}}
\newcommand{\cT}{\mathcal{T}}
\newcommand{\cD}{\mathcal{D}}
\newcommand{\fg}{\mathfrak{g}}
\newcommand{\fp}{\mathfrak{p}}
\newcommand{\fo}{\mathfrak{o}}
\newcommand{\fs}{\mathfrak{s}}
\newcommand{\fl}{\mathfrak{l}}

\theoremstyle{plain}
\newtheorem{theorem}{Theorem}[section]

\newtheorem{proposition}[theorem]{Proposition}
\newtheorem{corollary}[theorem]{Corollary}

\theoremstyle{definition}

\theoremstyle{remark}
\newtheorem{remark}[theorem]{Remark}

\numberwithin{equation}{section}

\title[Conformal Tractor Bundles]{Subtleties Concerning Conformal 
 Tractor Bundles}

\author{C. Robin Graham}
\address{Department of Mathematics, University of Washington,
Box 354350\\
Seattle, WA 98195-4350}
\email{robin@math.washington.edu}

\author{Travis Willse}
\address{Mathematical Sciences Institute, Building 27,
The Australian National University\\
Canberra ACT 0200 Australia}
\email{travis.willse@anu.edu.au}

\begin{document}

\maketitle

\thispagestyle{empty}

\renewcommand{\thefootnote}{}
\footnotetext{Partially supported by NSF grant \# DMS 0906035.}   
\renewcommand{\thefootnote}{1}

\section{Introduction}\label{intro}

The goal of this paper is to explain some phenomena arising in the 
realization of tractor bundles in conformal geometry as associated
bundles.  In order to form an associated bundle one chooses a 
principal bundle with normal Cartan connection 
(i.e., a normal parabolic geometry) corresponding to a given conformal 
manifold.  We show that different natural choices can lead to topologically
distinct 
associated tractor bundles for the same inducing representation.  The
nature of the choices is subtle, so we give a careful presentation of the
relevant foundational material which we hope researchers in the field will
find illuminating.  The main considerations apply as well to more general
parabolic geometries.    

We focus particularly on tractor bundles associated to the standard
representation of $O(p+1,q+1)$.  
The paper \cite{BEG} gave a construction of a canonical tractor bundle and 
connection on any conformal manifold $(M,c)$ which are now usually called
the standard tractor bundle and its normal tractor connection.    
This standard tractor bundle has other characterizations and realizations; 
these were studied in \cite{CG2}.  One of the realizations 
is as an associated bundle to a 
principal bundle over the conformal manifold; the 
standard tractor bundle is associated to the standard representation of
$O(p+1,q+1)$.  The complications arise because there are 
different ways to realize a given conformal manifold as a normal
parabolic geometry, corresponding to different choices of structure group 
and lifted conformal frame bundle.  Different such choices can give rise to
different   
tractor bundles with connection associated to the standard representation,
and for many natural choices one does not obtain the standard tractor
bundle with its normal connection.  For example, let 
$\cQ$ denote the model quadric for conformal geometry in signature
$(p,q)$, consisting of the space of null lines for a quadratic form of
signature $(p+1,q+1)$.  If one takes the 
homogeneous space realization $\cQ=O(p+1,q+1)/P^{\operatorname{line}}$, where
$P^\line$ denotes the isotropy group of a fixed null line, then for $pq\neq 
0$ the bundle associated to the standard representation of $P^\line$ is
not the standard tractor bundle.  Moreover its holonomy
(which is trival) is
not equal to the conformal holonomy of $\cQ$ (which is $\{\pm I\}$).     
Recall that $\cQ$ is orientable if its dimension $n=p+q$ is even, so for
$n$ even this phenomenon is not a consequence of failure of  
orientability of the conformal manifold.  The
issue is that this associated bundle does not have the correct topology. 

We begin by recalling in \S\ref{standard} the BEG construction of
the standard tractor bundle and its normal connection.  We then formulate 
a (slight variant of a) 
uniqueness theorem of \cite{CG2} providing conditions on a bundle with 
auxiliary data on 
a conformal manifold which characterize it as the standard tractor bundle
with its normal connection.  Following  
\cite{CG2}, we review the construction of a tractor bundle and connection
via the ambient 
construction and show using the \v{C}ap-Gover Uniqueness Theorem that this 
construction also produces the standard tractor bundle.  

In \S\ref{TCS} we review a fundamental prolongation result which we call
the TM\v{C}S Theorem (for Tanaka, Morimoto, \v{C}ap-Schichl), which asserts 
an equivalence of categories between certain categories of parabolic
geometries and categories of underlying structures on the base manifold.  
Our treatment is similar to that of \cite{CSl} except that we   
parametrize parabolic geometries and underlying structures by triples
$(\fg,P,\Ad)$, where $\fg$ is a $|k|$-graded semisimple Lie algebra, $P$ is
a Lie group with Lie algebra $\fp=\fg^0$, and $\Ad$ is a suitable
representation of $P$ on $\fg$.  Also we are more explicit about the
choices involved in determining an underlying structure.  
We then illustrate the TM\v{C}S Theorem by showing how it can be used to
represent general conformal manifolds and oriented conformal manifolds as
parabolic geometries.  In each case, in order to obtain a category of
parabolic geometries one must make a choice of a Lie group $P$ whose Lie
algebra is the usual parabolic subalgebra $\fp\subset \fs\fo(p+1,q+1)$,
and, depending on the 
choice of $P$, also a choice of a lift of the conformal frame bundle.  
There are several choices, some of which are equivalent.  

In \S\ref{trac} we describe the construction of tractor bundles and
connections as associated bundles for general parabolic geometries, 
and then specialize to the parabolic geometries arising from conformal
structures discussed in \S\ref{TCS}.  We parametrize our associated bundles
by suitably compatible $(\fg,P)$-modules; as for our parametrization of
parabolic geometries we find that this clarifies the dependence on the
various choices.  We make some observations about general tractor bundles
as associated bundles for conformal geometry, and then we specialize to the
question of which choices from \S\ref{TCS} give rise   
to the standard tractor bundle when one takes the $(\fg,P)$-module to be
the standard representation.  There are preferred choices for which one
always obtains the standard tractor bundle:  for conformal manifolds one
should choose $P^\ray$, the subgroup of $O(p+1,q+1)$ preserving a null ray,
and for oriented conformal manifolds one should choose $SP^\ray$, the
subgroup of $SO(p+1,q+1)$ preserving a null ray.  This is well-known and is 
often taken as the definition of the standard tractor bundle.  What is
novel in our discussion is the fact that so many other natural
choices give bundles associated to the standard representation which are
not the standard tractor bundle with its normal connection.     

In \S\ref{trac} we also briefly discuss homogeneous models and conformal
holonomy.  We follow the  
usual convention of defining the conformal holonomy of a conformal manifold 
to be the holonomy of the standard tractor bundle with its normal
connection, and show that for natural choices of principal bundles it often
happens that the 
holonomy of the tractor bundle with normal connection associated to the
standard representation is not equal to the conformal   
holonomy.  We conclude
\S\ref{trac} with a brief discussion of analogous 
phenomena for the parabolic geometries corresponding to generic 2-plane
fields on 5-manifolds, the consideration of which led us to become aware of
these subtleties in the first place.  

Throughout, our conformal structures are of signature $(p,q)$ on
manifolds $M$ of dimension $n=p+q\geq 3$.  

We are grateful to Andreas \v{C}ap and Rod Gover for useful comments and  
suggestions.  

\section{Standard Tractor Bundle}\label{standard}

The paper \cite{BEG} gave a concrete construction of a tractor bundle
$\cT$ on general conformal manifolds.  $\cT$ has rank $n+2$, carries a
fiber metric $h$ of signature 
$(p+1,q+1)$, and has a null rank 1 subbundle $\cT^1$ isomorphic to the
bundle 
$\cD[-1]$ of conformal densities of weight $-1$.  We denote by $\cD[w]$ the
bundle of conformal densities of weight $w$ and by $\cE[w]$ its space of
smooth sections.  The bundle $\cT$ was defined to be 
a particular conformally invariant subbundle of the 2-jet bundle 
of $\cD[1]$.  It was then shown that a choice $g$ of a representative of
the conformal class induces a splitting 
$$
\cT\cong\cD[-1]\oplus TM[-1]\oplus \cD[1],
$$
where 
$TM[w]=TM\otimes \cD[w]$.  With respect to this splitting, a section 
$U\in \Gamma(\cT)$ is represented as a triple
$$
U=
\begin{pmatrix}
\rho\\
\mu^i\\
\sigma
\end{pmatrix}
$$
with $\rho\in \cE[-1]$, $\mu^i\in \Gamma(TM[-1])$, $\sigma\in \cE[1]$.   
Under a conformal change $\gh=e^{2\Up}g$, the representations are
identified by
$$
\begin{pmatrix}
\wh{\rho}\\
\wh{\mu}^i\\
\wh{\sigma}
\end{pmatrix}
=
\begin{pmatrix}
1&-\Up_j&-\tfrac12 \Up_k\Up^k\\
0&\delta^i{}_j&\Up^i\\
0&0&1
\end{pmatrix}
\begin{pmatrix}
\rho\\
\mu^j\\
\sigma
\end{pmatrix}.
$$
Indices are raised and lowered using the tautological 2-tensor 
${\bf g}\in \Gamma(S^2T^*M[2])$ determined by the conformal structure. 
The tractor metric $h$ is defined by 
$$
h(U,U)= 2\rho\sigma + {\bf g}_{ij}\mu^i\mu^j.  
$$
The subbundle $\cT^1$ is defined by $\mu^i=0$, $\sigma=0$, and the map  
$$
\rho\mapsto 
\begin{pmatrix}
\rho\\
0\\
0
\end{pmatrix}
$$
defines a conformally invariant isomorphism $\cD[-1]\cong \cT^1$.  

Note for future reference that since conformal density bundles are trivial,
$\cT$ is isomorphic to  
$TM\oplus \R^2$ as a smooth vector bundle, where $\R^2$ denotes a trivial
rank 2 vector bundle.  It follows that $\cT$ is orientable if and only if
$M$ is orientable.

A connection $\nabla$ on $\cT$ was defined in \cite{BEG} directly in terms  
of the splitting and the chosen representative and the definition was
verified to be conformally invariant and to give $\nabla h=0$.  The
definition is: 
$$
\nabla_i
\begin{pmatrix}
\rho\\
\mu^j\\
\sigma
\end{pmatrix}
=
\begin{pmatrix}
\nabla_i \rho-P_{ik}\mu^k\\
\nabla_i \mu^j+\delta_i{}^j\rho+P_i{}^j\sigma\\
\nabla_i\sigma-\mu_i
\end{pmatrix}.
$$
The occurrences of $\nabla_i$ on the right-hand side denote the
connection induced by the representative $g$ on the density bundles, or
that connection coupled with the Levi-Civita connection of $g$ in the case
of $\nabla_i\mu^j$.  
$P_{ij}$ denotes the Schouten tensor of $g$.

A uniqueness theorem for such a tractor bundle 
was proven in \S 2.2 of \cite{CG2}.  We state the result
assuming a conformal manifold, whereas in \cite{CG2} the existence  
of the conformal structure was part of the conclusion.  Let $(M,c)$ be a 
conformal manifold with tautological tensor ${\bf g}\in
\Gamma(S^2T^*M[2])$. Consider the following data.  Let $\cT$ be  
a rank $n+2$ vector bundle over $M$ with metric $h$ of signature
$(p+1,q+1)$ and connection $\nabla$ such that $\nabla h=0$.  Let $\cT^1$ be
a null line subbundle of $\cT$ equipped with an isomorphism 
$\cT^1\cong\cD[-1]$.  If $v\in T_xM$ and $U\in \Gamma(\cT^1)$,
differentiating $h(U,U)=0$ shows that 
$\nabla_vU \in (\cT^1_x)^\perp$.  The projection of $\nabla_vU$
onto $(\cT^1_x)^\perp/\cT^1_x$ is tensorial in $U$.  Invoking 
the isomorphism $\cT^1\cong\cD[-1]$, it follows that 
$v\otimes U\mapsto \nabla_vU + \cT^1$ induces a bundle map $\tau:TM\otimes
\cD[-1]\rightarrow (\cT^1)^\perp/\cT^1$.  The 
metric $h$ determines a metric $h_0$ of signature $(p,q)$ on
$(\cT^1)^\perp/\cT^1$.  The  
data $(\cT,\cT^1,h,\nabla)$ are said to be compatible with the conformal
structure if $\tau^*h_0={\bf g}$.  We refer to \cite{CG2}     
for the formulation of the curvature condition for $\nabla$ to be called
normal.    

\bigskip
\noindent
{\bf \v{C}ap-Gover Uniqueness Theorem}.  {\it Let $(M,c)$ be a conformal
  manifold.   Up to 
  isomorphism, there is a unique $(\cT,\cT^1,h,\nabla)$ as above compatible
  with the conformal structure with $\nabla$ normal.   
}

\bigskip
Such a $\cT$ is called a (or the) standard tractor bundle and $\nabla$ its
normal tractor connection.  Even though $(\cT,\cT^1,h,\nabla)$ is unique up
to isomorphism, 
there are several different realizations.  The tractor bundle and
connection constructed in \cite{BEG} satisfy the conditions and so provide 
one realization.   

Another realization of the standard tractor bundle discussed in \cite{CG2}
is via the ambient 
construction of \cite{FG1}, \cite{FG2}.  Let $\cG\rightarrow M$ be the
metric bundle of $(M,c)$, i.e. 
$\cG=\{(x,g_x):x\in M, g\in c\}\subset S^2T^*M$.  $\cG$ carries 
dilations $\delta_s$ for $s>0$ defined by $\delta_s(x,g)=(x,s^2g)$, and the  
tautological 2-tensor ${\bf g}$ can be viewed as a section 
${\bf g}\in\Gamma(S^2T^*\cG)$ satisfying $\delta_s^*{\bf g}=s^2{\bf g}$.  
An ambient metric $\gt$ for $(M,c)$ is a metric of signature 
$(p+1,q+1)$ on a dilation-invariant neighborhood $\cGt$ of $\cG\times
\{0\}$ in $\cG\times \R$ 
satisfying $\delta_s^*\gt = s^2\gt$, $\iota^*\gt={\bf g}$, 
and a vanishing condition on its Ricci curvature.  (In order to construct
the standard tractor bundle and its normal connection, it suffices that the 
tangential components of the Ricci curvature of $\gt$ vanish when
restricted to $\cG\times \{0\}$.)  Here      
$\iota:\cG\rightarrow \cG\times \R$ is defined by $\iota(z)=(z,0)$ and the
dilations extend to $\cG\times \R$ acting in the $\cG$ factor alone.  

The ambient realization of $\cT$ is defined as follows.  The fiber $\cT_x$
over $x\in M$ is
$$
\cT_x= \left\{U\in\Gamma(T\cGt|_{\cG_x}):\delta_s^*U=s^{-1}U\right\},
$$
where $\cG_x$ denotes the fiber of $\cG$ over $x$.  
The homogeneity condition implies that $U$ is determined by its
value at any single point of $\cG_x$, so $\cT_x$ is a vector space of
dimension $n+2$.  
The tractor metric $h$ and the normal tractor connection $\nabla$ can be
realized as the restrictions to $\cG$ of $\gt$ and its Levi-Civita
connection $\nt$.  The null subbundle $\cT^1$ is
the vertical bundle in $T\cG\subset T\cGt|_\cG$.  The infinitesimal 
dilation $T$ defines a global section of $\cT^1[1]$, so determines the 
isomorphism $\cT^1\cong\cD[-1]$.  It can be verified that the tractor
bundle and connection defined this 
way satisfy the conditions above, so the uniqueness   
theorem implies that the ambient construction gives 
a standard tractor bundle with its normal connection.  An isomorphism  
with the realization in \cite{BEG} is written down directly in \cite{GW} in  
terms of a conformal representative.

We mention in passing that the formulation of the ambient construction in
\cite{CG2} appears to be more general than that above in that it allows an   
arbitrary ambient manifold $\cGt$ with a free $\R_+$-action containing
$\cG$ as a hypersurface.  But at least near $\cG$ there is no real gain in
generality:  if $\cGt$ admits a metric $\gt$ such that $\iota^*\gt={\bf
  g}$, then the normal bundle of $\cG$ in $\cGt$ is trivial so that near 
$\cG$, $\cGt$ is diffeomorphic to a neighborhood of $\cG\times \{0\}$ in
$\cG\times \R$.  This is because the 1-form dual to $T$ with respect to
$\gt$ gives a global nonvanishing section of $(T\cGt/T\cG)^*$.  

The third usual construction of the standard tractor bundle is as an
associated bundle to the Cartan bundle for the conformal structure.  We 
postpone discussion of this construction to \S\ref{trac}.

\section{Tanaka-Morimoto-\v{C}ap-Schichl Theorem}\label{TCS}

A fundamental result in the theory of parabolic geometries asserts 
an equivalence of 
categories between parabolic geometries of a particular type $(\fg,P)$
and certain underlying structures.  
There are different forms of the result due to Tanaka \cite{T},
Morimoto \cite{M}, and \v{C}ap-Schichl \cite{CSc}.  We state a version
which is a slight extension of Theorem 3.1.14 in \cite{CSl} and refer  
to it as the TM\v{C}S Theorem.     

Let $\fg=\fg_{-k}\oplus \cdots \oplus \fg_k$ be a $|k|$-graded
semisimple Lie algebra with associated filtration 
$\fg^i=\fg_{i}\oplus \cdots \oplus \fg_k$  and 
subalgebras $\fp=\fg^0$ and $\fg_-=\fg_{-k}\oplus \cdots \oplus \fg_{-1}$.
Let $P$ be a Lie group with Lie algebra $\fp$ and let   
$\Ad:P\rightarrow \Aut_{\operatorname{filtr}}(\fg)$ be a representation of  
$P$ as filtration-preserving Lie algebra automorphisms of $\fg$ such  
that $p\mapsto\Ad(p)|_\fp$ is the usual adjoint representation of $P$ on
$\fp$.  
Typically there is a Lie group $G$ with Lie algebra $\fg$ containing $P$ as
a parabolic subgroup with respect to the given grading, and $\Ad$ is the   
restriction to $P$ of the adjoint representation of $G$.  But we assume
neither that there exists such a $G$ nor that we have chosen one.  
For fixed $|k|$-graded $\fg$,  
two choices $(\fg,P_1,\Ad_1)$ and $(\fg,P_2,\Ad_2)$ will be   
regarded as equivalent from the point of view of the TM\v{C}S Theorem if
there is an isomorphism $\gamma:P_1\rightarrow P_2$ of Lie groups which
induces the identity on the common Lie algebra $\fp$ of $P_1$ and $P_2$ and
which satisfies $\Ad_2\circ \gamma = \Ad_1$.    

Given data $(\fg,P,\Ad)$ as above, the Levi subgroup $P_0\subset P$ is
defined by 
$$
P_0=\{p\in P: \Ad(p)(\fg_i)\subset \fg_i, -k\leq i\leq k\}.
$$
We prefer the notation $P_0$ rather than the usual $G_0$ since we do not
choose a group $G$, and also to emphasize that $P_0$ depends on $P$.

A parabolic geometry of type $(\fg,P,\Ad)$ (or just $(\fg,P)$ if the
representation $\Ad$ is understood) on a manifold $M$ is a
$P$-principal bundle $\cB\rightarrow M$ together with a Cartan connection 
$\omega:T\cB\rightarrow \fg$.  The definition of a Cartan connection depends
only on the data $(\fg,P,\Ad)$; see, for example, \cite{S}.    
We refer to \cite{CSl} for the conditions on   
the curvature of $\omega$ for the parabolic geometry to be called regular
and normal.   

Next we formulate the notion of an underlying structure of type
$(\fg,P,\Ad)$ on a manifold $M$.  The first part of the data consists of a
filtration 
$TM=T^{-k}M\supset \cdots \supset T^{-1}M\supset \{0\}$ of $TM$ 
compatible with the Lie bracket such that at each point $x\in M$ the
induced Lie algebra structure 
on the associated graded $\gr(T_xM)$ (the symbol algebra) is 
isomorphic to $\fg_-$.  We denote by $\cF(\fg_-,\gr(TM))$ the induced 
frame bundle of $\gr(TM)$ whose structure group is the group
$\Aut_{\gr}(\fg_-)$ of 
graded Lie algebra automorphisms of $\fg_-$ and whose fiber over 
$x$ consists of all the graded Lie algebra isomorphisms 
$\fg_-\rightarrow \gr(T_xM)$.  The second part of the data is a 
$P_0$-principal bundle $E\rightarrow M$ equipped with a bundle map 
$\Phi:E\rightarrow \cF(\fg_-,\gr(TM))$ covering the identity on $M$ which
is equivariant with respect to the homomorphism 
$\Ad:P_0\rightarrow \Aut_{\gr}(\fg_-)$ in the sense that 
$\Phi(u.p)=\Phi(u).\Ad(p)$ for $p\in P_0$, $u\in E$.  An underlying
structure of type $(\fg,P,\Ad)$ on $M$ is such a filtration of $TM$
together with such a $P_0$-principal bundle $E$ and map $\Phi$.  

There are notions of morphisms of parabolic geometries and of underlying
structures of type $(\fg,P,\Ad)$ which make these into categories.   A
morphism of parabolic 
geometries $\cB_1\rightarrow M_1$ and $\cB_2\rightarrow M_2$ of type
$(\fg,P,\Ad)$ is a principal bundle morphism $\phi:\cB_1\rightarrow \cB_2$ 
such that $\phi^*\omega_2=\omega_1$.  A morphism of underlying structures
$E_1\rightarrow M_1$ and $E_2\rightarrow M_2$ 
of type $(\fg,P,\Ad)$ is a principal bundle morphism $\phi:E_1\rightarrow
E_2$ which covers a filtration-preserving local diffeomorphism 
$f:M_1\rightarrow M_2$ and
which is compatible with the maps $\Phi_1$, $\Phi_2$ in the sense that 
$\Phi_2\circ\phi=f_*\circ\Phi_1$, where 
$f_*:\cF(\fg_-,\gr(TM_1))\rightarrow \cF(\fg_-,\gr(TM_2))$ is the map on
the frame bundles induced by the differential of $f$.  If $(\fg,P_1,\Ad_1)$
and $(\fg,P_2,\Ad_2)$ are equivalent from the point of view of the TM\v{C}S
Theorem as defined above, then composition of the principal bundle actions
with $\gamma$ 
induces an equivalence of categories between the categories of parabolic
geometries of 
types $(\fg,P_1,\Ad_1)$ and $(\fg,P_2,\Ad_2)$ and between the categories of 
underlying structures of types $(\fg,P_1,\Ad_1)$ and $(\fg,P_2,\Ad_2)$.  

The simplest and most common situation is when $\Ad:P_0\rightarrow
\Aut_{\gr}(\fg_-)$ is injective.  Then $\Phi$ is a bijection between $E$ 
and $\Phi(E)\subset\cF(\fg_-,\gr(TM))$, and $\Phi(E)$ is a subbundle of 
$\cF(\fg_-,\gr(TM))$ with structure group $\Ad(P_0)\cong P_0$.  It is not
hard to see that 
this association defines an equivalence between the category of underlying
structures of type $(\fg,P,\Ad)$ and the category of reductions of 
structure group of the frame bundle  
$\cF(\fg_-,\gr(TM))$ of filtered manifolds of type $\fg_-$ 
to $\Ad(P_0)\subset\Aut_{\gr}(\fg_-)$.  
In general, an underlying structure determines the bundle $\Phi(E)$, 
which is still a reduction of $\cF(\fg_-,\gr(TM))$ to structure group  
$\Ad(P_0)\subset\Aut_{\gr}(\fg_-)$.  
But if $\Ad:P_0\rightarrow \Aut_{\gr}(\fg_-)$ is not injective then   
the underlying structure contains more information.  In all the cases we
will consider, the kernel of  
$\Ad:P_0\rightarrow\Aut_{\gr}(\fg_-)$ is discrete.  Then 
$\Ad:P_0\rightarrow \Ad(P_0)$ and 
$\Phi:E\rightarrow \Phi(E)$ are covering maps, and to fix an underlying
structure one 
must also choose the lift $\Phi:E\rightarrow \Phi(E)$ of the reduced bundle
$\Phi(E)\subset \cF(\fg_-,\gr(TM))$ to a $P_0$-bundle.   

\bigskip
\noindent
{\bf TM\v{C}S Theorem}.  {\it Let $\fg$ be a $|k|$-graded semisimple Lie
  algebra 
such that none of 
the simple ideals of $\fg$ is contained in $\fg_0$, and such that
$H^1(\fg_-,\fg)^1=0$.  Let $P$ be a Lie group with Lie algebra $\fp$ and
let $\Ad:P\rightarrow \Aut_{\operatorname{filtr}}(\fg)$ be a representation
of $P$ as filtration-preserving Lie algebra automorphisms of $\fg$ such   
that $p\mapsto\Ad(p)|_\fp$ is the usual adjoint representation of $P$ on
$\fp$.  
Then there is an equivalence of categories between 
normal regular parabolic geometries of type $(\fg,P,\Ad)$ and underlying 
structures of type $(\fg,P,\Ad)$.  }

\bigskip
\noindent
Here $H^1(\fg_-,\fg)^1$ denotes the 1-piece in the filtration of
the first Lie algebra cohomology group.  All the examples we will consider 
satisfy $H^1(\fg_-,\fg)^1=0$.  

The discussion in \cite{CSl} is in terms of categories of parabolic
geometries and underlying structures of type $(G,P)$ and assumes from the
outset that that $P$ is a parabolic subgroup of $G$.   
Our point in formulating parabolic geometries and underlying structures of  
type $(\fg,P,\Ad)$ rather than type $(G,P)$ is not really to extend 
the discussion to the case that such a $G$ might not exist.  There is such
a group $G$ for all the examples we care about.  Rather, the point is 
to emphasize that the choice of a particular such $G$ is irrelevant as far
as the TM\v{C}S Theorem is concerned.  The fact that the TM\v{C}S Theorem
holds in the generality stated above has been communicated to us by Andreas  
\v{C}ap.  The emphasis on $\fg$ rather than $G$ and further
generalization in this direction are fundamental aspects of the work of 
Morimoto \cite{M}.           

Consider the case of general conformal structures of signature
$(p,q)$, $p+q=n\geq 3$.  The 
filtration of $TM$ is trivial, the frame bundle $\cF$ is the full frame
bundle of $M$, and a conformal structure is equivalent to a reduction of
the structure group of $\cF$ to  
$CO(p,q)=\R_+\, O(p,q)$.  The Lie algebra $\fg$ is $\fs\fo(p+1,q+1)$ and
we will consider various possibilities for $P$.  Take the quadratic  
form defining $\fs\fo(p+1,q+1)$ to be $2x^0x^\infty +h_{ij}x^ix^j$ for some
$h_{ij}$ of signature $(p,q)$.  Writing the matrices in terms of $1\times
n\times 1$ blocks, the Levi subgroups $P_0$ will be of the form  
\begin{equation}\label{P0}
P_0=\left\{p=
\begin{pmatrix}
\lambda&0&0\\
0&m&0\\
0&0&\lambda^{-1}
\end{pmatrix}
\right\}
\end{equation}
with various restrictions on $\lambda\in \R\setminus \{0\}$ and $m\in
O(p,q)$.  The matrices in $P$ 
will be block upper-triangular.  $\fg_-\cong \R^n$ consists of matrices of
the form  
$$
\begin{pmatrix}
0&0&0\\
x^i&0&0\\
0&-x_j&0
\end{pmatrix}
$$
with $x\in \R^n$, and $h_{ij}$ is used to lower the index.  Under the  
adjoint action, $p\in P_0$ acts on $\fg_-$ by $x\mapsto \lambda^{-1}mx$.    

A natural first choice is to take $P$ to be the subgroup $P^\ray$ of
$G=O(p+1,q+1)$ preserving the ray $\R_+ e_0$, with $\Ad$ the restriction of
the adjoint representation of $O(p+1,q+1)$.  Then $P^\ray_0$ is given by
\eqref{P0} with the restrictions $\lambda>0$, $m\in O(p,q)$.    
The map $\Ad:P_0^\ray\rightarrow CO(p,q)$ is an isomorphism, so an
underlying structure is exactly a conformal structure.   

There is another choice of $P$ which is equivalent to $P^\ray$ from the
point of view of the TM\v{C}S Theorem.  Namely, consider the subgroup 
$PP$ of $PO(p+1,q+1)=O(p+1,q+1)/\{\pm I\}$ preserving the line through
$e_0$ in the projective action, with $\Ad$ induced by the inclusion of 
$PP$ as a subgroup of $G=PO(p+1,q+1)$.  The coset projection
$P^\ray\rightarrow PP$ is an isomorphism which maps one $\Ad$
representation to the other, so these choices of $P$ are equivalent from
the point of view of the TM\v{C}S Theorem.

If $n$ is odd, there is yet another choice of $P$ also equivalent from
the point of view of the TM\v{C}S Theorem.  This is the subgroup $SP^\line$
of $SO(p+1,q+1)$ preserving the line through $e_0$, with $\Ad$ 
induced by the inclusion $SP^\line\subset O(p+1,q+1)$.  Observe that    
$SP^\line_0$ corresponds to $\lambda\neq 0$, $m\in SO(p,q)$.  If
$\lambda<0$, $m\in SO(p,q)$, and $n$ is odd, then $\lambda^{-1}m$ has
negative determinant, and one sees easily that 
$\Ad:SP_0^\line\rightarrow CO(p,q)$ is an isomorphism.  
The map $A\mapsto \det(A)A$ is an isomorphism from $P^\ray$ to $SP^\line$
which maps the $\Ad$ representation of $P^\ray$ on $\fs\fo(p+1,q+1)$ to
that of $SP^\line$.  Thus $P^\ray$ 
and $SP^\line$ are equivalent from the point of view of the TM\v{C}S
Theorem if $n$ is odd.

The TM\v{C}S Theorem asserts an equivalence of categories between conformal
structures and normal parabolic geometries of type
$(\fs\fo(p+1,q+1),P^\ray)$.  (The regularity condition is automatic for
conformal geometry since $\fs\fo(p+1,q+1)$ is $|1|$-graded.)  
Thus, for each conformal manifold, there is a $P^\ray$-principal bundle
$\cB^\ray$ carrying a normal Cartan connection, and it is unique up to 
isomorphism.  This is of course a classical result going back to Cartan.  
We will call this parabolic geometry of type $(\fs\fo(p+1,q+1),P^\ray)$ the
canonical parabolic geometry realization of the conformal manifold $(M,c)$.   
One may equally well choose to represent the canonical parabolic geometry 
using the structure group $PP$ (or $SP^\line$ if $n$ is 
odd), since these categories of parabolic geometries are equivalent.

It is instructive to identify the principal bundles for model $M$.  
For instance, suppose we take $M=S^p\times S^q$ to be the space of null
rays, with conformal structure determined by the metric $g_{S^p}-g_{S^q}$.  
We have $M=O(p+1,q+1)/P^\ray$, so the canonical parabolic geometry
realization determined by the TM\v{C}S Theorem is $\cB=O(p+1,q+1)$ with
Cartan connection  
the Maurer-Cartan form.  Alternately we can choose $M$ to be the   
quadric $\cQ=(S^p\times S^q)/\Z_2$ embedded in  
$\P^{n+1}$ as the set of null lines.  We can realize 
$\cQ=PO(p+1,q+1)/PP$ and view this as a parabolic geometry for $P^\ray$ via 
the isomorphism $P^\ray\cong PP$.  So the canonical parabolic geometry
realization of the quadric is  
$PO(p+1,q+1)$ with its Maurer-Cartan form as Cartan connection.  Thus 
in this sense both 
$PO(p+1,q+1)/PP$ and $O(p+1,q+1)/P^\ray$ are homogeneous models for  
the category of parabolic geometries of type $(\fs\fo(p+1,q+1),P^\ray)$.   

If $n$ is odd, then there is an alternate
homogeneous space realization of $\cQ$ using $SP^\line$; namely as
$\cQ=SO(p+1,q+1)/SP^\line$.  
The uniqueness assertion inherent in the TM\v{C}S Theorem imples that this 
realization must be isomorphic to that above.  Indeed, 
the map $A\mapsto \det(A)A$ determines an isomorphism 
$PO(p+1,q+1)\rightarrow SO(p+1,q+1)$ of the two parabolic geometry
realizations of $\cQ$.   

Next let us choose $P$ to be the subgroup $P^\line$ of $O(p+1,q+1)$
preserving the line spanned by the first basis vector $e_0$, with $\Ad$
induced by the inclusion in $G=O(p+1,q+1)$.  The Levi
factor $P^\line_0$ corresponds to 
the conditions $\lambda\neq 0$, $m\in O(p,q)$.  We have 
$\Ad(P^\line_0)=CO(p,q)$, so an underlying structure includes the data of a
conformal structure.  But now $\Ad$ is not injective; its kernel is $\{\pm
I\}$.  So to determine an underlying structure we must additionally choose
a lift $E$ of the conformal frame bundle to a $P_0^\line$-bundle.    
Such a lift always exists since $P_0^\line$ is a product:  
$P^\line_0\cong P_0^\ray\times \{\pm I\}$.  If $\cF_c$ denotes the
conformal frame bundle of $M$ with  
structure group $CO(p,q)\cong P_0^\ray$, then we can take 
$E=\cF_c\times \{\pm I\}$ with the product action of $P^\line_0$, and can
take the map $\Phi$ in the definition of underlying structures to be the
projection onto $\cF_c$.  Since $P^\line\cong P^\ray\times \{\pm I\}$, if 
$(\cB^\ray,\omega^\ray)$ denotes the canonical parabolic geometry
realization of the conformal manifold, then the bundle $\cB^\line$ produced 
by the TM\v{C}S Theorem 
for this choice of $E$ is just $\cB^\line=\cB^\ray\times \{\pm I\}$, 
and the Cartan connection is the pullback of $\omega^\ray$ to $\cB^\line$
under the obvious projection.    

However, depending on the topology of $M$, there may be a number of other 
inequivalent lifts $E$ 
of the conformal frame bundle to a $P^\line_0$-bundle, which will determine
inequivalent $P^\line$-principal bundles $\cB$ with  normal Cartan
connection via the TM\v{C}S Theorem.  For instance, consider the quadric
$\cQ$.  The product bundle 
constructed in the previous paragraph gives rise to the realization 
$\cQ=(PO(p+1,q+1)\times \{\pm I\})/P^\line$, with Cartan bundle
$\cB=PO(p+1,q+1)\times \{\pm I\}$.  On the other hand, the geometrically
obvious realization of  
$\cQ$ as a homogeneous space for $P^\line$ is as $\cQ=O(p+1,q+1)/P^\line$.
If $p=0$ or  
$q=0$, then $\cQ=S^n$ is simply connected so there is only one lift.  Indeed,
$O(n+1,1)\cong PO(n+1,1)\times \{\pm I\}$, corresponding to the
decomposition into time-preserving and time-reversing transformations.  But
if $pq\neq 0$, then $O(p+1,q+1)$ and $PO(p+1,q+1)\times \{\pm I\}$ are
inequivalent as $P^\line$-principal bundles over $\cQ$, as we will see in
the next section.  

There are analogues of all these choices for oriented conformal
structures.  In this case the structure group reduction is to
$CSO(p,q)=\R_+SO(p,q)\subset 
CO(p,q)$.  A natural choice is to take $P$ to be $SP^\ray$, the subgroup of 
$SO(p+1,q+1)$ preserving the null ray, for which $SP^\ray_0$ corresponds to 
$\lambda>0$, $m\in SO(p,q)$.  We have that $\Ad:SP^\ray_0\rightarrow
CSO(p,q)$ is an isomorphism, so in all dimensions and signatures underlying
structures of type $(\fs\fo(p+1,q+1),SP^\ray)$ are the same as oriented
conformal structures.  The parabolic geometry of type
$(\fs\fo(p+1,q+1),SP^\ray)$ determined by the TM\v{C}S Theorem is a
reduction to structure group $SP^\ray$ of the canonical parabolic geometry
of type $(\fs\fo(p+1,q+1),P^\ray)$ determined by the same conformal
structure but forgetting the orientation.    

If $n$ is even, for oriented conformal structures a choice of $P$ 
equivalent to $SP^\ray$ from the point of view of the TM\v{C}S   
Theorem is the subgroup $PSP$ of $PSO(p+1,q+1)$ preserving the null line in 
the projective action.  In all dimensions and signatures, a homogeneous
model is $S^p\times S^q= SO(p+1,q+1)/SP^\ray$.  The quadric $\cQ$ is
orientable if $n$ is even, and in this case it provides another homogeneous 
model for parabolic geometries of type
$(\fs\fo(p+1,q+1),SP^\ray)$:  $\cQ= PSO(p+1,q+1)/PSP$.  

For even $n$, a choice of $P$ for oriented conformal structures analogous 
to $P^\line$ above is $SP^\line$, since $\lambda^{-1}m$ remains
orientation-preserving for $\lambda<0$ if $n$ is even.   
For this choice the structure group reduction
is to $\Ad(SP_0^\line)=CSO(p,q)$.  But $\Ad$ has kernel $\{\pm I\}$, so a 
lift of the oriented conformal frame bundle must be chosen to determine an
underlying structure.  One has the
product decomposition $SP_0^\line\cong SP_0^\ray\times \{\pm I\}$, so one
choice is always the product lift.  But if $pq\neq 0$, the realization 
$\cQ = SO(p+1,q+1)/SP^\line$ corresponds to an inequivalent lift.

\section{Tractor Bundles as Associated Bundles}\label{trac}

Let $M$ be a manifold with a parabolic geometry $(\cB,\omega)$ of type
$(\fg,P,\Ad)$.   
There is an associated vector bundle $\cV\rightarrow M$ corresponding to
any finite-dimensional representation $\rho:P\rightarrow GL(V)$.  The
sections of $\cV$ can be 
identified with the maps $f:\cB\rightarrow V$ which are $P$-equivariant in
the sense that $R_p^*f=\rho(p^{-1})f$ for all $p\in P$.  
Suppose moreover that $(V,\rho)$ is actually a $(\fg,P)$-representation,
that is 
there is an action $\rho:\fg\rightarrow \fg\fl(V)$ of $\fg$ on $V$  
which is compatible with the $P$-action in the sense 
that the infinitesimal action of $\fp$ obtained by differentiating the
action of $P$ agrees with the restriction of the action of $\fg$ to $\fp$. 
We will say that the $(\fg,P)$-module $(V,\rho)$ is $\Ad$-compatible if 
$$
\rho\left(\Ad(p)(Z)\right)=\rho(p)\rho(Z)\rho(p^{-1})\qquad\quad 
p\in P,\;Z\in \fg.
$$
In this case there is an induced linear connection $\nabla$ on $\cV$
defined as follows.  Let $f$ be a section of $\cV$ and let $X$ be a vector
field on $M$.  Choose a lift ${\bar X}$ of $X$ to $\cB$ and set
\begin{equation}\label{connection}
\nabla_Xf={\bar X}f+\rho(\omega({\bar X}))f.
\end{equation}
The fact that $\omega$ reproduces generators of fundamental vector fields,
the equivariance of $f$, 
and the compatibility of the $(\fg,P)$-actions implies that the right-hand
side is unchanged upon adding a vertical vector field to ${\bar X}$.  Thus  
$\nabla_X f$ is independent of the choice of lift ${\bar X}$.  So one may
as well take ${\bar X}$ to be $P$-invariant.  Then ${\bar X}f$ is clearly
$P$-equivariant, and one checks easily 
that the $\Ad$-compatibility implies that $\rho(\omega({\bar X}))f$ is 
$P$-equivariant.  Thus $\nabla_Xf$ is a section of $\cV$.  The
resulting map $(X,f)\mapsto \nabla_Xf$ defines a connection on $\cV$.    
We call $(\cV,\nabla)$ the tractor bundle and tractor connection 
for the parabolic geometry $(\cB,\omega)$ associated to the
$\Ad$-compatible $(\fg,P)$-module $(V,\rho)$.      

As discussed in the previous section, typically one can find a Lie group
$G$ with Lie algebra $\fg$ which contains   
$P$ as a parabolic subgroup and which induces the given $\Ad$.
If $(V,\rho)$ is any finite-dimensional representation of $G$, the 
induced representations of $\fg$ and $P$ define a   
$(\fg,P)$-module structure which is automatically $\Ad$-compatible.  So
there is a tractor bundle and connection associated to any
finite-dimensional representation of any such group $G$.  We will call this
the tractor bundle and connection associated to the restriction to
$(\fg,P)$ of the $G$-module $(V,\rho)$.  

We saw in \S\ref{TCS} that one can choose different normal
parabolic geometries corresponding to a given conformal manifold $(M,c)$.
As we will see, different choices can give rise to different bundles
associated 
to the same $O(p+1,q+1)$-module $(V,\rho)$.  Recall that we have a 
canonical parabolic geometry corresponding to $(M,c)$:
the normal parabolic geometry of type $(\fs\fo(p+1,q+1),P^\ray)$.  Applying
the associated bundle construction for this choice gives a canonical   
tractor bundle and connection associated to any $O(p+1,q+1)$-module
$(V,\rho)$.    

We first compare tractor bundles and connections for the product parabolic  
geometry $(\cB^\line,\omega^\line)$ of type $(\fs\fo(p+1,q+1),P^\line)$
with those for the canonical parabolic geometry. 
Recall that the product parabolic geometry was defined as follows.
If $(\cB^\ray,\omega^\ray)$ is the canonical parabolic geometry, 
then $\cB^\line=\cB^\ray\times \{\pm I\}$ with the product action of  
$P^\line\cong P^\ray\times \{\pm I\}$, and $\omega^\line$ is the pullback
of $\omega^\ray$ under the projection $\cB^\line\rightarrow \cB^\ray$.  The 
bundle $\cB^\line$ may alternately be described as the $P^\line$-principal
bundle associated to the $P^\ray$-principal bundle $\cB^\ray$ by the action
of $P^\ray$ on $P^\line$ by left translation.   

Let $(V,\rho^\ray)$ be an $\Ad$-compatible
$(\fs\fo(p+1,q+1),P^\ray)$-module.  We will say that an $\Ad$-compatible 
$(\fs\fo(p+1,q+1),P^\line)$-module $(V,\rho^\line)$ extends $(V,\rho^\ray)$ 
if $\rho^\line=\rho^\ray$ on $\fs\fo(p+1,q+1)$ and 
$\rho^\line|_{P^\ray}=\rho^\ray$.  For a given $(V,\rho^\ray)$, there are
always at least two choices of such $\rho^\line$; namely those determined
by the two choices $\rho^\line(-I)=\pm I_V$.  One checks easily that either 
choice of $\pm$ defines an $\Ad$-compatible
$(\fs\fo(p+1,q+1),P^\line)$-module.   

\begin{proposition}\label{Pline}
Let $(M,c)$ be a conformal manifold.  Let $(V,\rho^\ray)$ be an
$\Ad$-compatible $(\fs\fo(p+1,q+1),P^\ray)$-module and let $(V,\rho^\line)$ 
be an $\Ad$-compatible $(\fs\fo(p+1,q+1),P^\line)$-module which extends 
$\rho^\ray$.  Then the tractor bundle and connection 
associated to the $(\fs\fo(p+1,q+1),P^\line)$-module $(V,\rho^\line)$  
for the product parabolic geometry $(\cB^\line,\omega^\line)$  
are naturally isomorphic to the tractor bundle and connection 
associated to the $(\fs\fo(p+1,q+1),P^\ray)$-module $(V,\rho^\ray)$
for the canonical parabolic geometry $(\cB^\ray,\omega^\ray)$.
\end{proposition}
 
\begin{proof}
We first claim that the bundle associated to 
$(V,\rho^\line)$ for $(\cB^\line,\omega^\line)$ 
is isomorphic to that associated to 
$(V,\rho^\ray)$ for $(\cB^\ray,\omega^\ray)$.  
This is a special case of the following 
general fact, the proof of which is straightforward.  
Suppose that $P_1$ is a Lie subgroup of a Lie group $P_2$ and $\cB_1$ is a
$P_1$-principal bundle over a manifold $M$.  Let $\cB_2$ be the
$P_2$-principal bundle over $M$ associated to the action of $P_1$ on $P_2$
by left 
translation.  Let $(V,\rho)$ be a $P_2$-module and $\cV_2$ the vector
bundle associated to $(V,\rho)$ for $\cB_2$.  Then $\cV_2$ is naturally
isomorphic as a smooth vector bundle to the vector bundle $\cV_1$
associated to $(V,\rho|_{P_1})$ for $\cB_1$.  

It is clear from \eqref{connection} that the tractor connections induced by
the Cartan connections $\omega^\ray$ and $\omega^\line$ correspond under
this isomorphism, since $\cB^\ray$ can be embedded as an open subset of 
$\cB^\line$ on which $\omega^\line$ restricts to $\omega^\ray$.   
\end{proof}

The following corollary is an immediate consequence of
Proposition~\ref{Pline}. 
\begin{corollary}\label{Plinecor}
If $(V,\rho)$ is an $O(p+1,q+1)$-module, then the tractor bundle and
connection for the product parabolic geometry $(\cB^\line,\omega^\line)$
associated to the restriction to $(\fs\fo(p+1,q+1),P^\line)$ of $(V,\rho)$  
are naturally isomorphic to the 
tractor bundle and connection 
for the canonical parabolic geometry $(\cB^\ray,\omega^\ray)$ 
associated to the restriction to
$(\fs\fo(p+1,q+1),P^\ray)$ of $(V,\rho)$.
\end{corollary}

The tractor bundle and connection associated to a given
$O(p+1,q+1)$-module for other normal parabolic geometry realizations of a 
conformal manifold may be different from those for the canonical parabolic  
geometry.  The basic case is the standard
representation $\V$ of $G=O(p+1,q+1)$, since any finite-dimensional
$O(p+1,q+1)$-module is isomorphic to a submodule of a direct sum
of tensor powers of the standard representation.  Consider first the
canonical parabolic geometry.    

\begin{proposition}\label{Pray}
Let $(M,c)$ be a conformal manifold.  Let $(\cB^\ray,\omega^\ray)$ be the 
canonical parabolic geometry of type $(\fs\fo(p+1,q+1),P^\ray)$.    
Let $(\cT,\nabla)$ be the bundle and connection associated to the
restriction to $(\fs\fo(p+1,q+1),P^\ray)$ of the standard 
representation $\V$ of $O(p+1,q+1)$.  Then $(\cT,\nabla)$ is a standard
tractor bundle with normal connection.
\end{proposition}
\begin{proof}
Since the action of $P^\ray\subset O(p+1,q+1)$ preserves the quadratic form
defining $O(p+1,q+1)$, it follows that the $S^2\V^*$-valued constant
function on 
$\cB$ whose value at each point is this quadratic form is
$P^\ray$-equivariant.  So  
it defines a metric $h$ on $\cT$.  Since the quadratic form  
is annihilated by $\fs\fo(p+1,q+1)$, the formula
\eqref{connection} for the connection shows that $\nabla h=0$.  Now  
$P^\ray$ acts on $e_0$ by multiplication by $\lambda>0$.
It 
follows that $e_0$ determines a nonvanishing global section of $\cT[1]$.  
This section determines a null rank 1 subbundle $\cT^1$ of $\cT$ together
with an  
isomorphism $\cT^1\cong \cD[-1]$.  The compatibility of the data with the 
conformal structure and the normality of $\nabla$ follow from the fact that
$\omega$ is the normal Cartan connection for the structure; 
see \cite{CG1}, \cite{CG2}.  Thus $(\cT,\nabla)$ possesses the structure
defining a standard tractor bundle with normal connection.   
\end{proof}

Recall that if $n$ is odd, then $P^\ray$ and $SP^\line$ are equivalent
from the point of view of the TM\v{C}S Theorem.  So by composing the
principal bundle action on $\cB^\ray$ with the inverse of the isomorphism  
$A\mapsto (\det A)A$ from $P^\ray$ to $SP^\line$, the canonical parabolic
geometry $(\cB^\ray,\omega^\ray)$ can be viewed
as a parabolic geometry of type $(\fs\fo(p+1,q+1),SP^\line)$.  So there is
a tractor bundle and connection associated to the restriction to 
$(\fs\fo(p+1,q+1),SP^\line)$ of the standard representation of
$O(p+1,q+1)$.  This tractor bundle can alternately be 
described as the bundle associated to the restriction to
$(\fs\fo(p+1,q+1),P^\ray)$ of the representation  $\det \otimes \V$ of
$O(p+1,q+1)$ for the canonical parabolic geometry. 

\begin{proposition}\label{SPline}
Let $(M,c)$ be a nonorientable odd-dimensional conformal manifold.   
Let $(\cB,\omega)$ be the corresponding normal parabolic geometry of type  
$(\fs\fo(p+1,q+1),SP^\line)$.  The tractor bundle with normal connection
associated to the 
restriction to $(\fs\fo(p+1,q+1),SP^\line)$ of the standard 
representation $\V$ of $O(p+1,q+1)$ is not a standard tractor bundle.   
\end{proposition}
\begin{proof}
The group $SP^\line$ preserves a volume form on $\V$.  There is
an induced nonvanishing volume form for the associated bundle, so
the tractor bundle is orientable.  But we saw in \S\ref{standard} 
that the standard tractor 
bundle constructed in \cite{BEG} is orientable if and only if $M$ is
orientable.  Since standard tractor bundles are unique up to isomorphism,
it follows that the associated bundle is not a standard tractor bundle if
$M$ is not orientable.
\end{proof}

Recall that the quadric $\cQ$ is not orientable if $n$ is odd and $pq\neq
0$.  Therefore we conclude:

\begin{corollary}\label{quadodd}
Let $n$ be odd and $pq\neq 0$.  Represent $\cQ=SO(p+1,q+1)/SP^\line$.  
The tractor bundle on $\cQ$ associated to the standard representation of
$SP^\line$ is not a standard tractor bundle.
\end{corollary}

Proposition~\ref{Pray} and Corollary~\ref{Plinecor} imply that for the
product parabolic geometry of type $(\fs\fo(p+1,q+1),P^\line)$ on a general
conformal manifold, the tractor
bundle and connection associated to the standard representation are a 
standard tractor bundle with its normal connection.  But as we saw in the
last section, there may be other normal parabolic   
geometries of type $(\fs\fo(p+1,q+1),P^\line)$ corresponding to the
same conformal structure, and the product parabolic geometry might 
not be the most geometrically natural choice.  For other choices the
associated tractor bundle need not be a standard tractor bundle.

\begin{proposition}\label{Plinequadric}
Represent $\cQ=O(p+1,q+1)/P^\line$.  If $pq\neq 0$, then 
the tractor bundle on $\cQ$ associated to the standard representation of
$P^\line$ is not a standard tractor bundle.
\end{proposition}
\begin{proof}
The null line subbundle $\cT^1$ is associated to the action of $P^\line$ on
the invariant subspace $\span\{e_0\}$.  It is easily
seen that this associated bundle is the tautological bundle, whose fiber at
a null line is the line itself.  But the tautological bundle on $\cQ$ is
not trivial if $pq\neq 0$, so it cannot be isomorphic to $\cD[-1]$.  
\end{proof}

\noindent
In Proposition~\ref{Plinequadric}, $\cT^1$ is associated to the
representation of $P^\line$ in which $p$ in \eqref{P0} acts by $\lambda$, 
while $\cD[-1]$ is associated to the representation in which $p$ 
acts by $|\lambda|$.  For this homogeneous space these associated 
bundles are not equivalent.

\begin{remark}
The tractor bundle on $\cQ$ in Proposition~\ref{Plinequadric} is trivial
since it is a bundle on a homogeneous space $G/P$ associated to the
restriction to $P$ of a representation of $G$.  In particular it is
orientable.  So if $n$ is odd, an alternate proof of 
Proposition~\ref{Plinequadric} is to derive a contradiction to
orientability of $\cT$ as in the proof of Proposition~\ref{SPline} and 
Corollary~\ref{quadodd}.  
But if $n$ is even, $\cQ$ is orientable and there is no contradiction to
orientability of $\cT$.  In this case the contradiction concerns the
orientability (equivalently, the triviality) of $\cT^1$, not of $\cT$.   

Observe also the following curious state of affairs in
Proposition~\ref{Plinequadric} when $n$ is odd.  The    
standard tractor bundle on $\cQ$ is nontrivial but its distinguished null
line subbundle is trivial.  By contrast, the associated tractor bundle
is trivial but its distinguished null line subbundle is nontrivial.     
\end{remark}

One encounters the same phenomena for oriented conformal structures.
Recall 
from the previous section that oriented conformal structures are equivalent
to normal parabolic geometries of type $(\fs\fo(p+1,q+1),SP^\ray)$.
The same  
proof as in Proposition~\ref{Pray} shows that the associated bundle for the
standard representation of $(\fs\fo(p+1,q+1),SP^\ray)$ is a standard
tractor bundle.  If $n$ is even, $SP^\line$ factors as 
$SP^\line=SP^\ray\times \{\pm I\}$, and the same proof as in
Proposition~\ref{Pline} shows 
that the associated bundle for the product
$SP^\line$-principal bundle is a standard tractor bundle.  But if $n$ is
even and the quadric is realized as $\cQ=SO(p+1,q+1)/SP^\line$, then the
proof of Proposition~\ref{Plinequadric} shows that 
the associated bundle to the standard representation of $SP^\line$ is not a
standard tractor bundle if $pq\neq 0$.

In the theory of Cartan geometries one sometimes declares a particular
connected homogeneous space $G/P$ to be the model,
and a tractor bundle on $G/P$ to be a bundle associated to the restriction
to $P$ of a representation of $G$.  Such tractor bundles are necessarily
trivial and the induced connection is the usual flat connection on a
trivial bundle.  For definite signature conformal
structures the natural choice is to take the model to be the quadric
$\cQ=S^n$, realized either as $O(n+1,1)/P^\line$ or as $O_+(n+1,1)/P^\ray$, 
where $O_+(n+1,1)$ denotes the   
time-preserving subgroup.  The above discussion shows that for either
realization the bundle associated to the standard representation of $G$ is  
the standard tractor bundle and its normal
connection is the induced flat connection.  For indefinite signature
conformal structures, a natural 
choice is to take the model to be $S^p\times S^q=O(p+1,q+1)/P^\ray$ and
again the same statements hold.  (The modification is necessary for
definite signature since $O(n+1,1)/P^\ray$ is not connected.)  It is also
possible to view the quadric as the homogeneous model in the case of
indefinite signature.  Since $PO(p+1,q+1)$ does not admit a standard
representation, the realization $\cQ=PO(p+1,q+1)/PP$ does not admit a
tractor bundle associated to the standard representation under 
the framework of this paragraph.  But the 
realization $\cQ=O(p+1,q+1)/P^\line$ does.  The general results stated
above of course remain true:  the 
bundle associated to the standard representation of $O(p+1,q+1)$ is trivial
and inherits the usual flat connection.  But this is not   
the standard tractor bundle:  the standard tractor bundle has no nontrivial 
parallel sections on $\cQ$.  One must exercise similar care in  
interpreting other results about homogeneous models.  For instance, it is a 
general result (\cite{CSS}) that on a homogeneous parabolic geometry $G/P$,
a BGG sequence resolves the constant sheaf determined by  
the inducing representation of $G$.  But one must keep in mind that the
bundles in the BGG sequences are all defined as associated bundles.  For
example, for the BGG sequence associated to   
the standard representation $\V$ on $\cQ=O(p+1,q+1)/P^\line$ with $pq\neq
0$, the constant sheaf $\V$ is realized as the global kernel of the first
BGG operator 
$\tf(\nabla^2+P)$ acting not on the bundle of densities $\cD[1]$, but on a 
twisted version thereof (the dual to the tautological bundle), and this 
twisted version arises as the projecting part of the associated bundle to
$\V$, which is not the standard tractor bundle for the conformal structure
on $\cQ$.  The global kernel of $\tf(\nabla^2+P)$ acting on $\cD[1]$ for
$\cQ$ is trivial if $pq\neq 0$.

Similar issues arise in the consideration of conformal holonomy.  
We follow the usual practice of defining the conformal holonomy of a
conformal manifold to be the holonomy of   
a standard tractor bundle with its normal tractor connection.  This is 
well-defined by the \v{C}ap-Gover Uniqueness Theorem.  But the above 
considerations demonstrate that one must be careful if one is realizing the 
standard tractor bundle as an associated bundle.  The holonomy of a
tractor bundle defined as an associated bundle to a 
standard representation might not equal the conformal holonomy if the
principal bundle is not chosen correctly.  This
happens already for the quadric $\cQ$ if $pq\neq 0$.  If we realize 
$\cQ=O(p+1,q+1)/P^\line$, then the tractor bundle associated to the
standard representation of $(\fs\fo(p+1,q+1),P^\line)$ has trivial holonomy.
But as discussed above, the
standard tractor bundle of $\cQ$ is the bundle associated to the standard 
representation of $(\fs\fo(p+1,q+1),P^\ray)$ for the realization
$\cQ=PO(p+1,q+1)/PP$, viewed as a parabolic geometry for
$(\fs\fo(p+1,q+1),P^\ray)$ via the isomorphism $P^\ray\cong PP$. 
Its holonomy is $\{\pm I\}$, since parallel
translation in $S^p\times S^q$ to the antipodal point induces $-I$ on a
fiber of the standard tractor bundle on $\cQ$.   Another instance of this 
is the following.
\begin{proposition}\label{hol}
Let $(M,c)$ be a nonorientable odd-dimensional conformal manifold.  Recall
from the TM\v{C}S Theorem that up to isomorphism there is a unique
normal  parabolic geometry of type  
$(\fs\fo(p+1,q+1),SP^\line)$ corresponding to $(M,c)$.  Let $(\cT,\nabla)$
be the bundle and connection associated to the 
restriction to $(\fs\fo(p+1,q+1),SP^\line)$ of the standard representation
of $O(p+1,q+1)$.  Then the holonomy of $(\cT,\nabla)$ is not equal to the 
conformal holonomy.     
\end{proposition}
\begin{proof}
The argument is similar to the proof of 
Proposition~\ref{SPline}.  The standard representation of $SP^\line$
preserves a volume form, so there is an induced section of the
associated bundle.  The volume form is preserved also by $\fs\fo(p+1,q+1)$, 
so this section is parallel.  Thus the holonomy of the associated bundle
for $SP^\line$ is contained in $SO(p+1,q+1)$.  But if the holonomy of  
the standard tractor bundle is contained in $SO(p+1,q+1)$, then the
standard tractor bundle is orientable, so $M$ is orientable.  
\end{proof}

These issues
concerning standard tractor bundles as associated bundles arose in our work
\cite{GW} concerning conformal structures and ambient metrics of holonomy
$G_2$ (by this we mean the split real form throughout this discussion).   
Nurowski \cite{N} showed that a generic 2-plane field on a 5-manifold $M$ 
induces a conformal structure of signature $(2,3)$ on $M$.  The TM\v{C}S
Theorem implies that generic 2-plane 
fields on oriented 5-manifolds are the underlying structures corresponding
to normal regular parabolic geometries of type $(\fg_2,SQ^\ray)$, where
$SQ^\ray$ is the  
subgroup of $G_2$ preserving a null ray, analogous to $SP^\ray$ above.  For
generic 2-plane fields on nonorientable manifolds one must change the group
$P=SQ^\ray$ to allow   
orientation-reversing transformations in $\Ad(P_0)$.  A first guess is to
take $P$ to be $SQ^\line$, the subgroup of $G_2$ preserving a null line.  
The TM\v{C}S Theorem implies that the category of generic 2-plane fields on
general 5-manifolds is equivalent to the category of normal regular
parabolic geometries of type $(\fg_2,SQ^\line)$.  But just as in
Proposition~\ref{SPline}, the associated bundle to the restriction to
$(\fg_2,SQ^\line)$ of the standard representation of $G_2$ need not be the
standard tractor bundle for Nurowski's induced conformal structure; in
fact, it cannot be if $M$ is not orientable.  
By analogy with the situation above for general conformal structures, 
instead of $SQ^\line$ one should use the subgroup $Q^\ray$ of 
$\{\pm I\}G_2$ 
preserving a null ray.  The TM\v{C}S Theorem again gives an equivalence of
categories with normal regular parabolic geometries of type
$(\fg_2,Q^\ray)$.  And now, just as in Proposition~\ref{Pray}, 
the tractor bundle associated to the restriction to $(\fg_2,Q^\ray)$ of the
standard representation of $\{\pm I\}G_2\subset O(3,4)$ is the standard
tractor bundle of Nurowski's conformal structure with its normal connection.   
$Q^\ray$ is isomorphic to $SQ^\line$, but they are embedded in $O(3,4)$
differently, just as for $P^\ray$ and $SP^\line$ above.  Using the 
realization of the standard tractor 
bundle as the associated bundle for $(\fg_2,Q^\ray)$, the same  
arguments as in \cite{HS}, \cite{GW} for the orientable case now show that 
for general $M$, Nurowski's conformal structures are characterized by
having conformal    
holonomy contained in $\{\pm I\}G_2$, and in the real-analytic case the
corresponding ambient metrics have metric holonomy contained in 
$\{\pm I\}G_2$.

\end{document}